\documentclass[12pt]{amsart}
\usepackage{amssymb,amsmath,amscd,graphicx,
latexsym,amsthm}
\usepackage{amssymb,latexsym,eufrak,amsmath,amscd,graphicx}
  \usepackage[all]{xy}
  \usepackage{pdfsync}
\theoremstyle{plain}
\newtheorem{theorem}{Theorem}
\newtheorem{proposition}[theorem]{Proposition}
\newtheorem{corollary}[theorem]{Corollary}
\newtheorem{lemma}[theorem]{Lemma}

\newtheorem{proposition.definition}[theorem]{Proposition/Definition}

\newtheorem{theoremalpha}{Theorem}
\newtheorem{corollaryalpha}[theoremalpha]{Corollary}
\newtheorem{propositionalpha}[theoremalpha]{Proposition}
\newtheorem{conjecturealpha}[theoremalpha]{Conjecture}

\theoremstyle{definition}

\newtheorem{remark}[theorem]{Remark}
\newtheorem{example}[theorem]{Example}

\newtheorem{problem}[theorem]{Problem}

\newcommand{\lra}{\longrightarrow}

\newcommand{\noi}{\noindent}
\newcommand{\PP}{\mathbf{P}}

\newcommand{\RR}{\mathbf{R}}

\newcommand{\ZZ}{\mathbf{Z}}

\newcommand{\QQ}{\mathbf{Q}}

\newcommand{\OO}{\mathcal{O}}

\newcommand{\eps}{\varepsilon}

\newcommand{\rk} {\text{rank }}

\newcommand{\HH}[3]{H^{{#1}} \big( {#2} , {#3}
\big) }

\newcommand{\hh}[3]{h^{{#1}} \big( {#2} , {#3}
\big) }

\newcommand{\tn}[1]{\textnormal{#1}}

\newcommand{\vol}{\textnormal{vol}}

\newcommand{\eee}{\mathbf{e}}
\newcommand{\rX}{\mathrm{X}}
\newcommand{\rY}{\mathrm{Y}}
\newcommand{\rB}{\mathrm{B}}
\newcommand{\rZ}{\mathrm{Z}}
\newcommand{\Var}{\textnormal{Var}}

\numberwithin{theorem}{section}

\begin{document}

\title{Asymptotics of Random Betti Tables}

 \author{Lawrence Ein}
  \address{Department of Mathematics, University of Illinois at Chicago, 851 South Morgan St., Chicago, IL  60607}
 \email{{\tt ein@uic.edu}}
 \thanks{Research of the first author partially supported by NSF grant DMS-1001336.}
 
 \author{Daniel Erman}
  \address{Department of Mathematics, University of Michigan, Ann Arbor, MI
   48109}
   \email{{\tt erman@umich.edu}}
\thanks{Research of the second author partially supported by the Simons Foundation.}

 \author{Robert Lazarsfeld}
  \address{Department of Mathematics, University of Michigan, Ann Arbor, MI
   48109}
 \email{{\tt rlaz@umich.edu}}
 \thanks{Research of the third author partially supported by NSF grant DMS-1159553.}

\maketitle

\setlength{\parskip}{.20 in}

 \section*{Introduction}

 The purpose of this paper is twofold. First, we present a conjecture to the effect that the ranks of the syzygy modules of a smooth projective variety become normally distributed as the positivity of the embedding line bundle grows. Then, in an attempt to render the conjecture plausible, we  prove a result suggesting that this is in any event the typical behavior from a probabilistic  point of view. Specifically, we consider a  ``random" Betti table with a fixed number of rows, sampled according to a uniform choice of Boij-S\"oderberg coefficients. We compute the asymptotics of the entries as the length of the table goes to infinity,  and show that they  become normally distributed with high probability.  
  
 Turning to details, we start by discussing at some length the geometric questions underlying the present work. Let $X$ be a smooth projective variety of dimension $n$ defined over some field $k$, and for   $d > 0$ put \[   L_d \ = \ dA + P \]
 where $A$ is a fixed ample divisor and $P$ is an arbitrary divisor on $X$. We assume that $d$ is sufficiently large so that $L_d$ defines a normally generated embedding
 \[   X \, \subseteq \,   \PP^{r_d}, \ \ \text{where } \ r \, = \, r_d \, =_{\text{def}} \, h^0(X, L_d)-1\, = \,  O(d^n). \]
 Write $S = \text{Sym}\,\HH{0}{X}{L_d}$ for the homogeneous coordinate ring of $\PP^{r_d}$, denote by $J = J_X \subseteq S$ the homogeneous ideal of $X$, and consider the minimal graded free resolution $E_{\bullet} = E_{\bullet}(X; L_d)$ of $J$:
 \[
 \xymatrix{
0 &J_X \ar[l]& \oplus S(-a_{1,j}) \ar[l]  \ar@{=}[d]& \oplus S(-a_{2,j})  \ar[l] \ar@{=}[d]& \ar[l] \ldots & \ar[l] \oplus S(-a_{r,j}) \ar[l] \ar@{=}[d] &  \ar[l]0 . \\ 
&  & E_1 & E_2 & &E_r }
 \]
 As customary, it is convenient to define
 \[
K_{p,q} \big ( X; L_d \big) \ = \ \Big \{
\parbox{1.9in}{\begin{center} minimal generators of $E_p(X;L_d)$ of degree $p + q$\end{center}}
\Big  \}.  
\]
  Thus $K_{p,q}(X;L_d)$ is a finite-dimensional vector space, and
\[  E_p(X;L_d) \ = \  \underset{q}{\bigoplus}   \ K_{p,q}(X;L_d) \, \otimes _k \,  S(-p-q). \]
We refer to an element of $K_{p,q}$ as a $p^{\text{th}}$ syzygy of weight $q$. 
 The dimensions
 \[   k_{p,q}(X; L_d) \ =_{\text{def}} \dim \, K_{p,q}(X;L_d) \]
 are the \textit{Betti numbers} of $L_d$; they are the entries of the \textit{Betti table} of $L_d$.  The basic problem motivating the present paper (one that alas we do not solve) is to understand the asymptotic behavior of these numbers as $d \to \infty$.
 
 Elementary considerations of Castelnuovo-Mumford regularity show that if $d \gg 0$ then
 \[  K_{p,q}(X; L_d) \ = \ 0 \ \ \text{for }  \ q > n+1.\]
 Furthermore $K_{p,n+1} (X; L_d) \ne 0$  if and only if
 \[     r_d - n - (p_g - 1) \ \le p \le \ r_d - n, \] 
 where $p_g = h^0(X, \omega_X)$. So the essential  point is to understand the groups $K_{p,q}(K, L_d)$ for $1 \le q \le n$ and $p \in [1, r_d]$. The main result of \cite{ASAV} is that as $d \to \infty$ these groups become non-zero for ``essentially all" values of the parameters. Specifically, there exist constants $C_1, C_2 > 0$ (depending on $X$ and the choice of the  divisors $A, P$ appearing in the definition of $L_d$) such that if one fixes $1 \le q \le n$, then
 \[ K_{p,q}(X; L_d) \ \ne \ 0 \]
 for every value of $p$ satisfying 
 \[   C_1 \cdot d^{q-1} \ \le \ p \ \le \ r_d - C_2 \cdot d^{n-1}. \]
However the results of \cite{ASAV} do not say anything quantitative about  the asymptotics in $p$ of the corresponding Betti numbers $k_{p,q}(X; L_d)$ for fixed weight $q \in [1,n]$ and $d \gg 0$.

The question is already interesting in the case $n = 1$ of curves: here it is only  the  $k_{p,1}$ that come asymptotically into play.\footnote{If $X$ is a curve of genus $g >2$ and $L_d$ is a divisor of degree $d \gg 0$ on $X$,  then the Betti numbers $k_{p,1}(X; L_d)$ depend on the geometry of $X$ and $L_d$ --  in a manner that is not completely understood -- when $r_d - g \le p \le r_d-1$. However elementary estimates show that when $d$ is large the invariants in question are exponentially smaller than the $k_{p,1}$ for $p \approx \frac{r_d}{2}$, and so they do not show up in the asymptotic picture. See Remark \ref{Curve.Variability.Remark}.}   Figure 1 shows plots of these Betti numbers for a divisor of degree $75$ on a curve of genus $0$, and on a curve of genus $10$.      The figure suggests that the $k_{p,1}$ become normally distributed, and we prove that  this is indeed the case:     
\vskip 10 pt
\begin{figure} [th!]\label{Betti.Picture}
   \includegraphics[scale = 1]{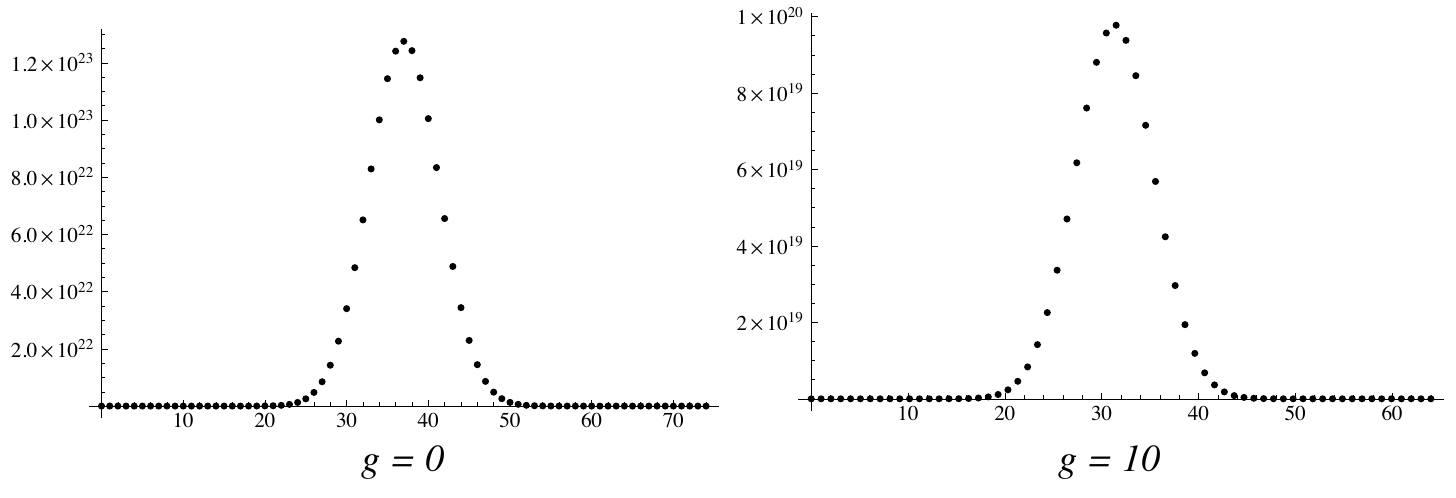}
   \vskip -10pt
   \caption{Betti numbers $k_{p,1}$ on curves of genus $0$ and $10$}
   \vskip 10pt
  \end{figure}

\begin{propositionalpha} \label{Betti.Asymptotics.Curves}
   Let $L_d$ be a divisor of large degree $d$ on a smooth projective curve $X$ of genus $g$, so that $r_d = d - g$. Choose a sequence $\{ p_d \}$  of integers such that 
\[ p_d \ \to\ \frac{r_d}{2} + a \cdot \frac{\sqrt{r_d}}{2}\]
for some fixed number $a$ \textnormal{(}ie.\ $\lim_{d \to \infty} \,  \frac{2p_d - r_d}{\sqrt{r_d}} = a$\textnormal{)}.  Then as $d \to \infty$,
\[   
\left ( \frac{1}{2^{r_d} }\cdot \sqrt{\frac{2\pi}{r_d}}\, \right ) \cdot k_{{p_d},1}(X;L_d)  \to e^{-a^2 / 2}. \]
   \end{propositionalpha}

    At the risk of some recklessness, we conjecture that the picture seen in dimension one holds universally: 
        \begin{conjecturealpha} \label{Betti.Asymptotics.Conj}
  Let $X$ be a smooth projective variety of dimension $n$, and fix a weight $1 \le q \le n$. Then there is a normalizing function $F_q(d)$ $($depending on $X$ and geometric data$)$ such that\[
F_q(d)\cdot k_{p_d,q}(X;L_d) \lra e^{-a^2 / 2} 
\]
as $d \to \infty$ and $p_d \to  \frac{r_d}{2} \ + \ a \cdot \frac{\sqrt {r_d}}{2} $.
   \end{conjecturealpha} 
   \noi In other words,  the prediction is that as $d \to \infty$ one gets the same sort of normal distribution of the Betti numbers $k_{p,q}(X;L_d)$, considered as functions of $p$, as that which occurs in the case of curves. Put another way, the conjecture asserts that the rows of the Betti table  of any very positive embedding display roughly the  pattern that  one would see in a large Koszul complex. 
          
             As an illustration, we plot in Figure \ref{4Fold.Veronese.P2} the Betti numbers $k_{p,1}$ for the 
 embedding $\PP^2 \subset \PP^{14}$ defined by $\OO_{\PP^2}(4)$, which is the largest example we were able to run on \texttt{Macaulay2}. We hope that the reader will agree that  these data  at least seem  consistent with the conjectured picture. 
 \begin{figure} 
 \vskip5pt
 \includegraphics[scale = .9]{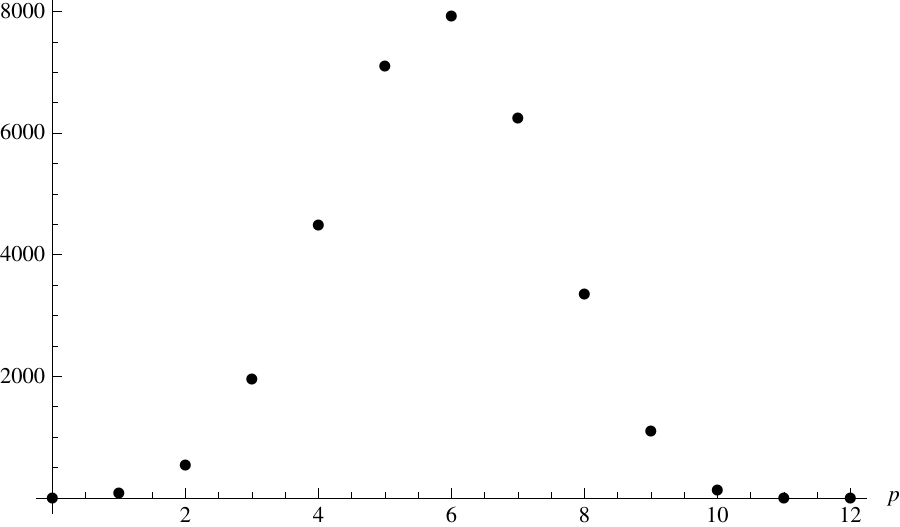}  
      \caption{Betti numbers $k_{p,1}$ of $4$-fold Veronese embedding of $\PP^2$.}
    \label{4Fold.Veronese.P2}
  \end{figure}

   Concerning the Conjecture, the first point to stress is that we don't know how to verify it for any variety of dimension $n \ge 2$. For example, it already seems a very interesting and  challenging problem to compute  the asymptotics of the Betti numbers of the $d$-fold Veronese embeddings of $\PP^2$. 
  In view of this state of affairs, it then becomes natural to ask whether one can give any indirect evidence supporting the conjecture. For example, if one considers  resolutions with syzygies having fixed weights $1 \le q \le n$ and lets the length   of the resolution grow, are the asymptotics predicted by the conjecture ``typical" in some sense? Our main purpose here is to prove a result suggesting  that this is indeed the case: with high probability, the entries in the rows of a ``random" Betti table become normally distributed as the length of the table goes to infinity.

  In order to make this  precise we start by introducing a  model for the resolutions in question;  we will  then apply the Boij-S\"oderberg theory established by Eisenbud and Schreyer \cite{ES} to construct a sample space for studying random Betti tables.  Specifically, fix a  natural number $n\ge 2$ and let  $R = k[x_1, \ldots, x_{r-n}]$ be a polynomial  ring in $r-n$ variables, where $r$ is a large integer that will later go to infinity. Consider now a finite-length graded $R$-module $M$ with the property that  
  \begin{equation}  \label{Bounding.Eqn} K_{p,q}(M) \ = \ 0 \ \ \text{for } q \not \in [1,n]    \end{equation}
and for every $0 \le p \le n-r$. We think of $M$ as having a resolution that models in a slightly simplified manner the resolutions that occur for very positive embeddings of smooth projective varieties of dimension $n$.   In the geometric setting, if $X$ carries a line bundle $B$ such that $H^i(X, B) = 0 $ for all $i$ -- for instance \ $B = \OO_{\PP^n}(-1)$ on $X  = \PP^n$ --  then  modules as in (*) can be constructed by starting with the graded $S$-module associated to $B$ and   modding out by $n+1$ general linear forms. 

  The theory of Eisenbud-Schreyer asserts that the Betti numbers $k_{p,q}(M)$ are determined by the Betti
   numbers of  modules having a \textit{pure resolution}. By definition, this is a module $\Pi$ with the property that the $p ^{\text{th}}$ syzygies of $\Pi$ all occur in a single weight $q = q(p)$. Pure modules satisfying \eqref{Bounding.Eqn} are then described by an $(n-1)$-element subset $I\subseteq [r] =_{\text{def}} \{1,\ldots, r\}$ encoding the  values of $p$ at which the weight $q(p)$ jumps: we denote by $\binom{[r]}{n-1}$ the collection of all such $I$. The corresponding module $\Pi_I$ is not unique, but its Betti numbers $k_{p,q}(\Pi_I)$ are  determined up to a scalar that is normalized by fixing the multiplity of  each $\Pi_I$. The main result of \cite{ES} implies that given $M$ as above, there exist non-negative rational numbers $x_I \in \QQ$ such that 
  \begin{equation}   \label{BS.Decomp.Eqn} k_{p,q}(M) \ = \ \sum_{I \in \binom{[r]}{n-1}} x_I \cdot k_{p,q}(\Pi_I)  \end{equation}
   for all $ p \in [0,r-n]$ and $q \in [1,n]$.\footnote{The set of indices $I$ that arise here form in a natural way the vertices of   a simplicial complex, and if one takes into account the simplicial structure one can arrange that the expression in \eqref{BS.Decomp.Eqn} is unique. However for reasons that we will discuss in Remark \ref{Ignore.Simplex.Remark}, we prefer to allow the indices $I$ to vary independently. Therefore when $n\ge 3$ the coefficients $x_I$ are not uniquely determined by $M$.} 
 In other words, the resolution of $M$ is numerically a $\QQ$-linear combination of the resolutions of the $\Pi_I$. Conversely, after possibly scaling one can find a module $M$ that realizes given non-negative rational numbers $x_I$. 
Thus up to scaling, the possible numerical types of $n$-weight resolutions are parametrized by vectors 
\[
x \ =\ \{ x_I  \}_{I \in \binom{[r]}{n-1}}\] 
of Boij-S\"oderberg coefficients $x_I \in \QQ_{\ge 0}$.\footnote{As explained in the previous footnote, this parametrization involves some repetitions.}  In order to emphasize the dependence on $r$, which will shortly become important, we will henceforth write $x_r = \{ x_{r, I} \} $ to denote these coefficients and $\Pi_{r, I}$ to denote the corresponding module. 

We assume now (by scaling) that each $x_{r,I} \le 1$, and since we are interested in numerical questions we allow the $x_{r,I}$ to be real. Denote by \[
\Omega_r \ = \ \Omega_{r,n} \ =_{\text{def}} \  [0,1]^{\binom{r}{n-1}}
\]
the cube parametrizing the resulting coefficient vectors $x_r= \{ x_{r,I} \}$. Then
given
\[  x_r \ = \ \{ x_{r, I} \} \ \in  \ \Omega_{r} ,  \]
set
\[  k_{p,q}(x_r) \ = \ \sum_{I \in \binom{[r]}{n-1}} \, x_{r,I} \cdot k_{p,q}(\Pi_{r,I}). \]
 Thus the $k_{p,q}$ are functions on $\Omega_{r}$ computing the Betti numbers of a module described by a Boij-S\"oderberg coeffient vector $x_r$.\footnote{In the following, we choose normalizations in such a way that each $\Pi_{r,I}$ has formal multiplicity $=1$.  There may be no actual module with this property, 
 but since we are working only up to scale, this doesn't cause any problems.}

We next imagine chosing $x_r \in \Omega_r$ uniformly at random.  The resulting real numbers $k_{p,q}(x_r)$ can then be thought of as  the entries of a random (and hence ``typical") Betti table with $n$ rows and $r +1-n$ columns. This  is illustrated  when $n = 2$ in Figure \ref{Two.Random.Diagrams}, which shows plots of $k_{p,1}(x_r)$ for    random   vectors  $x_{r} \in \Omega_{r}$ with $r = 14$ and $r=60$. Fixing $q \in [1,n]$, our main result describes   the  asymptotics in $p$ of the numbers $k_{p,q}(x_r)$ for such a randomly chosen coefficient vector $x_r$ as $r$ gets very large. It implies in particular that when $r \to \infty$, the Betti numbers $k_{p,q}(x_r)$ become normally distributed with high probability.  

 \vskip 10 pt
\begin{figure} 
   \includegraphics[scale = .85]{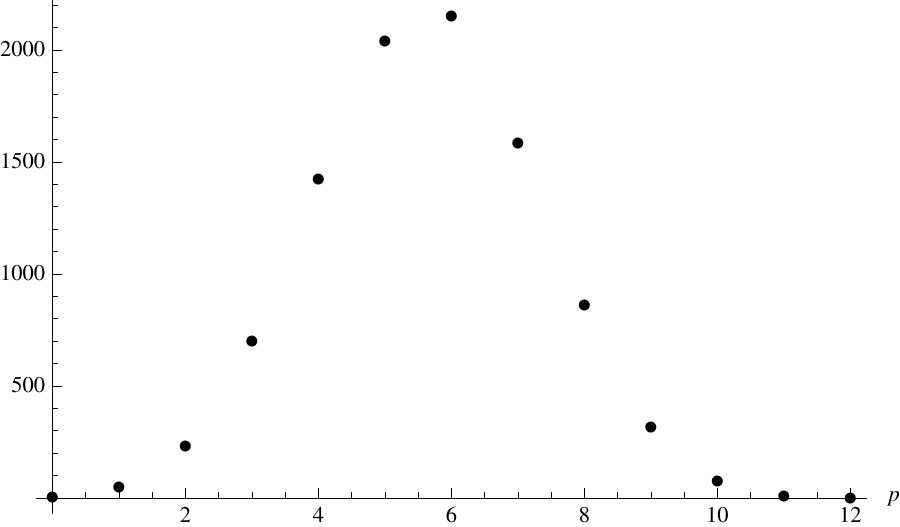} \qquad  \includegraphics[scale = .85]{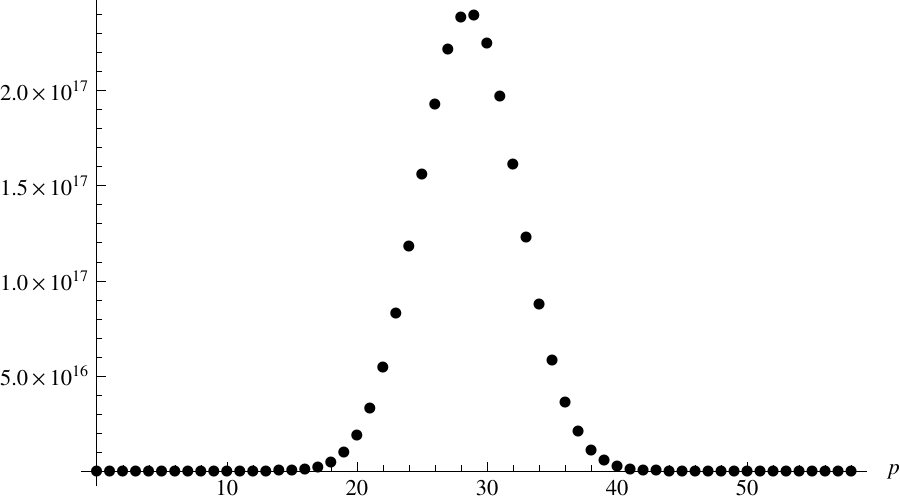}
   \caption{Betti numbers $k_{p,1}$ of random Betti tables with $r = 14$ and with $r=60$.} \label{Two.Random.Diagrams}
  \end{figure}

 We start with a somewhat informal statement.
 \begin{theoremalpha} \label{Main.Theorem.Intro}
 Fix a weight $q \in [1,n]$. If one chooses $x_r \in \Omega_r$ randomly, then as $r \to \infty$ one expects
 \begin{equation} \label{Asymptotic.Formula}
 k_{p,q}(x_r)  \ \sim \ \frac{1}{2^n} \cdot \binom{r-n}{p} \cdot \left ( \frac{p^{q-1}\cdot (r-p-n)^{n-q}}{ (q-1)! \cdot (n-q)! }  \right) \end{equation}
provided that $\frac{p}{r}$ is bounded away from $0$ and $1$.  
 \end{theoremalpha}
 \noi More precisely,  given $\eps > 0$ denote by
 \[  \Sigma_{p,r,n}(\eps) \ \subseteq \ \Omega_{r} \]
 the set of all coefficient vectors $x_r = \{ x_{r,I}\} \in \Omega_r$ such  that
 \[
 \left | \frac{k_{p,q}(x_r) }{ \big ( \textnormal{RHS of \eqref{Asymptotic.Formula}} \big)}  \, - \, 1 \right | \ > \ \eps.
 \]
 Now suppose that $\{ p_r \}$ is a sequence of positive integers $0 \le p_r \le r-n$ with the property that 
 \[  c \, < \, \frac{p_r}{r} \, < \, 1 - c \]
 for some small $c > 0$. Put the standard  Lebesgue measure on the cube $\Omega_r$, so that $\vol( \Omega_r)=1$. Then the assertion of the Theorem is  that
 \[  
 \lim_{r \to \infty} \,   \vol \big( \Sigma_{p_r,r,n}(\eps)  \big) \,  \ = \ 0, \]
 and that moreover the convergence is uniform in $p_r$ (given the bounding constant $c$). 
 
 The crucial term in \eqref{Asymptotic.Formula} is the binomial coefficient.   Stirling's formula then implies that as functions of $p$, the $k_{p,q}(x)$ display the sort of normal distribution predicted by the Conjecture with high probability. 
 \begin{corollaryalpha} \label{Stirling.Cor.Intro}
 Fix $n$ and $q \in [1,n]$. There exists a function $F(r) = F_{q,n}(r)$
 with the property that if $\{ p_r \}$ is a sequence with
 \[
 p_r \ \to \ \frac{r}{2} \, + \, a \cdot \frac{\sqrt{r}}{2} \]
for some $a$,  then 
\[
F(r) \cdot k_{p_r,q}(x_r) \, \to \, e^{-a^2/2} 
\]
for ``most" choices of $x_r \in \Omega_r$.
\end{corollaryalpha}
\noi We refer to \S 2 for the precise meaning of the last phrase. 
 
Returning to the  setting of Conjecture \ref{Betti.Asymptotics.Conj},   one most likely wouldn't  expect the modules arising from geometry to be described by  the sort of uniform choice of  Boij--S\"oderberg coefficients built into  Theorem \ref{Main.Theorem.Intro} and Corollary \ref{Stirling.Cor.Intro}. However it seems that the qualitative picture coming out of our model is rather robust, and persists for other probability measures on $\Omega_r$. We prove a specific result in this direction in \S  \ref{Weight.Function.Subsection}.

The sort of probabilistic model discussed in the present paper is very
different from the notion of genericity that traditionally comes up in commutative algebra.  Specifically, the standard perspective is to consider moduli spaces parametrizing flat families of graded modules (or
graded free resolutions) with fixed numerical invariants, and to examine what happens at a general point of this moduli space.   The  random Betti tables studied here  behave quite differently than what one expects in the algebraic setting. For example, these two approaches display essentially opposite semicontinuity properties.  A random Betti table in our sense will be drawn from the interior of the cone of Betti tables, and thus ``more
random Betti table = less sparse Betti table."  By contrast, in the classical moduli perspective  each individual Betti number is lower semicontinuous, and thus ``more generic resolution = more sparse Betti
table."  We analyze this tension more explicitly in a particular case in \S \ref{Generic.Mxs.Subsection}. From this viewpoint, Conjecture \ref{Betti.Asymptotics.Conj} predicts that the resolutions arising from geometry are extremely non-generic algebraically.\footnote{It is not surprising that this should be so. For instance, the ideal of a very positive embedding of any variety is generated by quadrics, but of course the subspace spanned by these quadrics sits in very special position within the linear system of all quadrics on the ambient projective space.}

We close with a philosophical remark. Based on experience with the case of curves, it was expected in some quarters that high degree embeddings of algebraic varieties would display rather sparse Betti tables. However the results of \cite{ASAV} showed that this is not at all the case. On the other hand, the results of that paper do suggest that the syzygies of very positive embeddings  exhibit uniform asymptotic behavior. If one believes this, it  becomes somewhat  hard to guess -- given the computations here for random Betti tables -- what one might expect other than  a statement along the lines of Conjecture \ref{Betti.Asymptotics.Conj}.

Concerning the organization of the paper, in \S \ref{Random.Diagrams.Section}  we introduce random Betti diagrams and carry out the  computations leading to Theorem \ref{Main.Theorem.Intro}. These are recast using Stirling's formula in \S \ref{Betti.Asymptotics.Section}, where we prove  Corollary \ref{Stirling.Cor.Intro}
 and Proposition \ref{Betti.Asymptotics.Curves}. Finally, \S \ref{Complements.Section}
 is devoted to some variants and open problems. 

The computer algebra programs \texttt{Macaulay2} \cite{Macaulay} and \texttt{Mathematica} \cite{Mathematica} provided useful assistance in studying examples.  We also thank David Eisenbud, Milena Hering and Claudiu Raicu for valuable suggestions. We have particularly profitted from discussions with Kyungyong Lee, who helped us to understand the direction this project should take.

   \section{Random Betti Tables} \label{Random.Diagrams.Section} 
 \setcounter{equation}{0}  
   \numberwithin{equation}{section}

   In this section we introduce some notation and definitions concerning Betti tables, and carry out the main computations.
   
We fix once and for all a positive integer $n\ge 2$, and denote by $r$ a large natural number that will eventually go to infinity. Generally speaking we will render   dependence on $r$ visible in the notation, but leave dependence on $n$ implicit.  We shall be concerned with $n \times (r+1-n)$ matrices of real numbers: the columns will be numbered $0, \ldots, r-n$ and the rows  $1, \ldots , n$. As in the Introduction such  \textit{Betti diagrams} arise upon tabulating the Betti numbers of a finite-length graded module over the polynomial ring $k[x_1, \ldots, x_{r-n}]$, but  in fact we will only be concerned with the tables themselves. Given a Betti table $B$, we denote by $k_{p,q}(B)$ the entry of $B$ in the $q^{\text{th}}$ row and $p^{\text{th}}$ column. (Some authors  work instead with 
$\beta_{p,j}(B)= k_{p,j-p}(B)$.)
   
   The theorem of Eisenbud-Schreyer \cite{ES} implies that the cone of all Betti tables of modules   as in the Introduction is spanned by so-called \textit{pure diagrams} $\pi(r,I)$ constructed as follows. Fix an $(n-1)$-element subset
   \[   I \ = \ (i_1 < \ldots < i_{n-1} ) \]
   of  $[r] =_{\text{def}} \{ 1, \ldots ,r\}$, and write
   \[ \{ 1, \ldots , r \} - I \ = \ \{ d_0, \ldots, d_{r-n} \}, \]
   with $d_0 < d_1< \ldots < d_{r-n}$. Then $\pi(r,I)$ is the Betti diagram with 
   $k_{p,j} = 0$ for $j \ne d_p-p$, and
  \begin{equation} \label{pure.diagram.I}
    k_{p, d_p-p}\big( \pi(r,I) \big )\ = \ (r-n)! \cdot \left( \prod_{\ell \ne p} \frac{1}{| d_{\ell} - d_p |} \right ). \end{equation}
   In other words, $\pi(r, I)$ is (up to scaling) the Betti table of a module all of whose syzygy modules are generated in a single degree. The integers $( d_0 < \ldots < d_{r-n})$ are the \textit{degree sequence} of $\pi(r,I)$, and the subset $I \subseteq [r]$ determines the values of $p$ at which the degree sequence  skips an integer. The numerical factors in \eqref{pure.diagram.I} are such that $\pi(r, I)$ has (formal) multiplicity $= 1$. (This follows from the formula
for the multiplicity of a module with a pure resolution; see, for
instance,~\cite[p.~88]{BS}.)
   
   \begin{example}
Consider the table $\pi(7,\{2,4\})$.  This corresponds to the degree sequence
\[
\{1,2,3,4,5,6,7\}\setminus \{2,4\}=
\{1,3,5,6,7\},\]
and one has:
\[
\pi(7,\{2,4\})=\bordermatrix{
&0&1&2&3&4\cr
1&\frac{1}{10}&-&-&-&-\cr
2&-&\frac{1}{2}&-&-&-\cr
3&-&-&\frac{3}{2}&\frac{8}{5}&\frac{1}{2}
}. \qed
\]
\end{example}
   
   It will be useful to have an alternative expression for the entries of $\pi(r, I)$. As a matter of notation, for $m \in \ZZ$, write 
   \[  m_+ \ = \ \max \, \{ m, 0 \}. \]
    Then \eqref{pure.diagram.I}
 yields the following, whose proof we leave to the reader:
   \begin{lemma}  \label{pure.diagram.II}
   The entries of $\pi(r, I)$ are given by
   \begin{equation}   k_{p,q} \big( \pi (r, I) \big) \ = \ (r-n)! \cdot \frac
   {\prod_{\alpha = 1}^{q-1} ( p + q - i_{\alpha})_+  \cdot \prod_{\alpha = q}^{n-1} (  i_{\alpha}-p-q)_+}{(p+q-1)!\cdot (r-p-q)!}. \ \qed \end{equation}
  
   \end{lemma}
   
   We next consider random Betti tables. For each $I \in \binom{[r]}{n-1}$, let $\rX_I$ be a random variable uniformly distributed on $[0,1]$. We take the $\rX_I$ to be independent. Then  put
   \begin{equation}  \rB_r \ = \ \rB_{r, n} \ = \ \sum_{I \in \binom{[r]}{n-1}} \, \rX_I \cdot \pi(r,I).
   \end{equation}
  Heuristically, one thinks of $\rB_{r}$ as the Betti table obtained by choosing  independent random Boij-S\"oderberg coefficients $x_I \in [0,1]$.  More formally, $\rB_{r}$ is a random $n \times (r+1-n)$ matrix, and for each $p,q$ the entry $k_{p,q}(\rB_{r})$ is a random variable. We write $\mathbf{e}_{p,q}(\rB_{r})$ for the expected value of this variable, i.e.  \begin{equation}
\mathbf{e}_{p,q}(\rB_{r}) \ =_{\text{def}} \ \mathbf{E} \big( k_{p,q} (\rB_{r}) \big).   
   \end{equation}

  \begin{remark} \label{Ignore.Simplex.Remark}
  When $n = 2$ the pure tables are indexed by a single integer $i \in [1,r] = \binom{[r]}{1}$. In this case the $\pi(r,i)$ are linearly independent, and there is a unique way to express every Betti table as a linear combination of pure diagrams. However when $n \ge 3$ chains in the set of degree sequences determine the simplices of a simplicial complex whose structure must be taken into account to get a unique decomposition. (See \cite{BS} or \cite{ES} for further details.) Here we have chosen to ignore this additional structure. In large part this is motivated just  by the desire to simplify the statements and computations. However there is emerging evidence from other directions that it can be advantageous to deal with all possible pure Betti tables, instead of just collections from a single simplex.  For instance, the recent work  \cite{gibbons} provides a pleasingly simple description of a pure table decomposition of any complete intersection, but this description relies on a collection of pure Betti tables that do not come from a single simplex.  
  \end{remark}
  
Our main result computes $\eee_{p,q}(\rB_{r})$, and shows that if we fix $q$ and let $r$ tend to infinity, then $k_{p,q}(\rB_{r,n})$ converges (up to an essentially polynomial factor)  to a binomial  distribution in $p$.

\begin{theorem} \label{Expected.Value.Thm}
 The expected entries of $\rB_{r}$ are given by the formula
\begin{equation} \label{Expected.Value.Equation}
\eee_{p,q}(\rB_{r}) \ = \ \frac{1}{2^n} \cdot \binom{r-n}{p} \cdot \left ( \frac{p^{q-1}\cdot (r-p-n)^{n-q}}{ (q-1)! \cdot (n-q)! } + o\big( p, r-p)^{n-1} \right) , 
\end{equation}
where the error term the right indicates a quantity that is $o(p^i (r-p)^j)$ for some $ i + j = n-1$. \end{theorem}

\begin{remark} Note that we could replace the factor $(r-p-n)$ in \eqref{Expected.Value.Equation} by $(r-p)$ without changing the statement. However we prefer to emphasize the  symmetry between $(p,q)$ and $(r-n-p, n+1-q)$ inherent in the situation.
\end{remark}

With $n$ and $q$ fixed, write
\begin{equation}
\label{Def.mu}   \mu(r,p) \ = \ \mu_{q,n}(r,p) \ = \ \frac{1}{2^n} \cdot \frac{p^{q-1}\cdot (r-p-n)^{n-q}}{ (q-1)! \cdot (n-q)! } \end{equation}
for the expression appearing on the right in \eqref{Expected.Value.Equation}.
The law of large  numbers then implies that $k_{p, q}(\rB_{r,n})$ converges to $\eee_{p,q}(\rB_{r,n})$. In fact:
\begin{corollary} \label{Convergence.Corollary} Fix a weight  $q \in [1,n]$, and let $\{p_r \}$
  be any sequence of positive integers $ 0 \le p_r \le r-n $ such that 
  \[ c \ \le \ \frac{p_r}{r} \ \le 1-c \] for some small $c > 0$. Then as $r \to \infty$,
\[ \frac{ k_{p_r,q}(\rB_{r}) }{  \binom{r-n}{p_r} \cdot \mu(r,p_r) } \longrightarrow 1\]
in probability, and moreover the convergence is uniform in $p_r$ \tn{(}for given $c$\tn{)}.  
\end{corollary}
\noi  In other words, given $\eps, \eta > 0$, one can find an integer $R$ such that if $r \ge R$ then
\[
\mathbf{P} \left ( \left | \frac{ k_{p_r,q}(\rB_{r}) }{  \binom{r-n}{p_r} \cdot \mu(r,p_r) } - 1\right | > \eps \right) \ < \ \eta,\]
and furthermore one can take $R$ to be independent of $p_r$ provided that $c < p_r/r < 1-c$.\footnote{Here and subsequently, we are writing $\PP$ to denote the probability of an event.} A more concrete interpretation of this assertion is spelled out  following the statement of Theorem \ref{Main.Theorem.Intro}
in the Introduction.

Turning to the proof of Theorem \ref{Expected.Value.Thm}, we begin with an  elementary lemma. As a matter of notation, given positive integers $a \ge b$, and a $b$-element subset $J \in \binom{[a]}{b}$ of $[a]$, write $J = \{ j_1 < \ldots < j_b \}$  for the elements of $J$ in increasing order.
\begin{lemma} \label{Simplex.Lattice.Pts.Lemma}
For fixed $b \ge 0$, consider the function $\sigma_b : \ZZ_{\ge b} \lra \ZZ$ given by
\[  \sigma_b(a) \ = \ \sum_{J \in \binom{[a]}{b}} \, j_1j_2 \ldots j_b. \]
Then
\[   \sigma_b(a) \ = \ \frac{1}{b!} \cdot \frac{a^{2b}}{2^b} \ + \ S_{2b-1}(a), \]
where $S_{2b-1}(x) \in \QQ[x]$ is a polynomial of degree $2b - 1.$\end{lemma}

\begin{proof}
Observe to begin with that
\[
\sum_{J \in \binom{[a]}{b}} j_1 j_2 \cdots j_b \ = \ \frac{1}{b!} \cdot\sum_{j_1, \ldots, j_b   \text{ distinct} }   j_1 j_2 \cdots j_b.
\]
On the other hand
\begin{equation}
\sum_{j_1, \ldots, j_b   \text{ distinct} } j_1 j_2 \cdots j_b  \ = \ \left( \sum_{j_1 , \ldots, j_b  \in [1,a]} j_1 j_2 \cdots j_b \right) \ - \ \left(\sum_{\text{two or more $j_k$ coincident}}   j_1 j_2 \cdots j_b\right).
\tag{*} \end{equation}
But first term on the right in (*) is \[\prod_{k = 1}^b \left ( \sum_{j_k = 1}^a j_k \right) \ = \ \left( \frac{(a+1)a}{2} \right) ^b, \]
while by induction on $b$ (and descending induction on the number of coincidences) the second term is of the form $R_{2b-1}(a)$ for some $R_{2b-1}(x) \in \QQ[x]$ of degree $2b-1$.  The Lemma follows.
\end{proof}

   \begin{proof}[Proof of Theorem \ref{Expected.Value.Thm}]
 Throughout the proof we fix $n$ and $q \in [1,n]$.      Note first from Lemma \ref{pure.diagram.II} that
       \small \begin{align*}
k_{p,q}(\pi(r,I))
&=\left(\frac{(p+q-i_1)_+\cdots (p+q-i_{q-1})_+\cdot (i_q-p-q)_+\cdots (i_{n-1}-p-q)_+}{(p+q-1)\cdots (p+1)\cdot (r-p-q)\cdots (r-p-(n-1))}\right)\binom{r-n}{p}.
\end{align*}
 \normalsize
  Given $I \in \binom{[r]}{n-1}$ define
  \small
   \begin{equation} \label{Def.YI}
\rY_{I,r,p}\ = \ \left(\frac{(p+q-i_1)_+\cdots (p+q-i_{q-1})_+\cdot (i_q-p-q)_+\cdots (i_{n-1}-p-q)_+}{(p+q-1)\cdots (p+1)\cdot (r-p-q)\cdots (r-p-(n-1))}\right)\cdot \rX_I.
\end{equation}
\normalsize
Thus for fixed $p, q$ and $r$:
\begin{equation}\label{kpq.from.Y}
\frac{k_{p,q}(\rB_r) }{\binom{r-n}{p} \cdot \mu(r,p)} \ = \ \frac{ \sum_{I \in \binom{[r]}{n-1}} \rY_{I,r,p}}{\mu(r,p)} ,
\end{equation}
where  $\mu(r,p)$ is the quantity from  \eqref{Def.mu}.

Now since $\mathbf{E}(\rX_I) = \frac{1}{2}$,
 for  fixed $r,p$ and $I$, we have:
 \small
\[ 
\mathbf{E}(\rY_{I,r,p})\ =\ \frac{1}{2}\cdot \left(\frac{(p+q-i_1)_+\cdots (p+q-i_{q-1})_+\cdot (i_q-p-q)_+\cdots (i_{n-1}-p-q)_+}{(p+q-1)\cdots (p+1)\cdot (r-p-q)\cdots (r-p-(n-1))}\right).
\]
\normalsize
Hence for fixed $p$ and $r$:
\small
\begin{align*}
\eee_{p,q}(\rB_{r,n}) 
&=\ \binom{r-n}{p}\cdot \sum_{I \in \binom{[r]}{n-1}} \mathbf{E}\left(\rY_{I,r,p}\right)\\
&= \ \binom{r-n}{p}\cdot \frac{1}{2} \cdot \frac{\sum_{I\in \binom{[r]}{n-1}} (p+q-i_1)_+\cdots (p+q-i_{q-1})_+ \cdot (i_q-p-q)_+\cdots (i_{n-1}-p-q)_+}{ (p+q-1)\cdots (p+1)\cdot(r-p-q)\cdots (r-p-(n-1))}\\
&=\ \binom{r-n}{p}\cdot \frac{1}{2}\cdot\frac{\left(\sum_{J\in \binom{[p+q-1]}{q-1}} j_1j_2\cdots j_{q-1}\right) \left( \sum_{J'\in \binom{[r-p-q]}{n-q}} j'_1\cdots j'_{n-q}\right)}{(p+q-1)\cdots (p+1)\cdot(r-p-q)\cdots (r-p-(n-1))}\\
&=\ \binom{r-n}{p}\cdot \frac{1}{2}\cdot \left( \frac{\sum_{J\in \binom{[p+q-1]}{q-1}} j_1j_2\cdots j_{q-1}}{(p+q-1)\cdots (p+1)}\right)\cdot\left( \frac{\sum_{J'\in \binom{[r-p-q]}{n-q}} j'_1\cdots j'_{n-q}}{(r-p-q)\cdots (r-p-(n-1))}\right).
\end{align*}
\normalsize
 Lemma \ref{Simplex.Lattice.Pts.Lemma} then implies  
 that
\begin{equation} \label{Big.Eqn}
\eee_{p,q} (\rB_{r}) \ = \ \binom{r-n}{p} \cdot \mu(r,p) \cdot \left(
 \frac{
 1 + \frac{A_{2q-3}(p)}{p^{2q-2}}}{ 1 + \frac{B_{2q-3}(p)}{p^{2q-2}}}\right)  \cdot  \left(
 \frac{
 1 + \frac{C_{2n-2q-1}(r-p)}{(r-p)^{2n - 2q}}}{ 1 + \frac{D_{2n-2q-1}(r-p)}{(r-p)^{2n - 2q }}}\right),
\end{equation}
where $A$, $B$, $C$, and $D$ are polynomials of the indicated degrees whose coefficients depend on $n$ and $q$. The Theorem follows.
\end{proof}
   
   \begin{remark} \label{Unif.Convergence.Remark}
  For later reference, we record a consequence of  the proof just completed. Specifically,  suppose given a sequence $\{ p_r \}$ with 
 \[  
 c \ \le \ \frac{p_r}{r} \ \le 1 - c
 \] 
 for some $c > 0$. Then it follows from \eqref{Big.Eqn}  that
 \begin{equation} \label{Exp.Value.Sum.Y.Eqn} \lim_{r \to \infty} \, \left( \frac{ \sum \mathbf{E}(\rY_{I, r,p_r})}{\mu(r,p_r) } \right) \ = \ 1, \end{equation}
 and that moreover the convergence is uniform in $p_r$ given $c$.
 \end{remark}
   
   \begin{proof}[Proof of Corollary \ref{Convergence.Corollary}.]
   This is essentially just the weak law of large numbers, but for keeping track of the dependence in $p$ it is quickest to go through the argument leading to that result, as in \cite[\S 5.1]{Chung}. As before, we fix $n$ and $q$ at the outset. Note to begin with that the coefficient of $\rX_I$ in \eqref{Def.YI} is $\le 1$, and hence
   \[   \Var (\rY_{I, r, p}) \ \le \ \frac{1}{3}  \]
   for every $I$, $r$ and $p$. Furthermore, since the $\rX_I$ are independent, for given $r$ and $p$ the $\rY_{I,r,p}$ are uncorrelated.  
Therefore, with $r$ and $p$ fixed  and $\delta > 0$,  Chebyschev's inequality  yields:
  \begin{align*}
 \mathbf{P} \left(\,  \left | \frac{ \sum \rY_{I,r,p} - \mathbf{E}\left( \sum \rY_{I, r, p} \right) }{\mu(r,p)} \right | \, > \, \delta \, \right ) \ &\le \ \frac{1}{\delta^2 \cdot \mu(r,p)^2} \, \cdot \, \Var \left( \sum \rY_{I,r, p} \right) \\ &\le \ \frac{1}{\delta^2 \cdot \mu(r,p)^2} \cdot \frac{\binom{r}{n-1} }{3} .\end{align*}
 Now fix a sequence $\{ p_r \}$ with $\frac{p_r}{r}$ bounded away from $0$ and $1$. Then   
 \[
 C_1 \cdot r^{n-1} \ \le \ \mu(r, p_r) \ \le \ C_2 \cdot r^{n-1}
 \]  for suitable positive constants $C_1, C_2$  independent of $p_r$. Hence there is a constant $C_3$, independent of $p_r$, such that 
 \begin{equation} \label{Prob.Bound}
 \mathbf{P} \left(\,  \left | \frac{ \sum \rY_{I,r,p_r} - \mathbf{E}\left( \sum \rY_{I, r, p_r} \right) }{\mu(r,p_r)} \right | \, > \, \delta \, \right ) \ \le \ \frac{C_3}{\delta^2 \cdot r^{n-1}} . \end{equation}
 Note that as $r \rightarrow \infty$, the term on the right $\rightarrow 0$.  Moreover, as we saw in Remark \ref{Unif.Convergence.Remark}
   \begin{equation} 
   \lim_{r \to \infty}  \frac{\mathbf{E}\left( \sum \rY_{I, r, p_r} \right ) }{\mu(r,p_r) } =  1, \notag \end{equation}
   uniformly in $p_r$. It follows that given $\eps > 0$:
   \[ 
\lim_{r \to \infty}   
\ \mathbf{P} \left( \ \left | \frac{ \mathbf{E}( \sum \rY_{I,r,p_r})}{\mu(r, p_r)}  - 1  \right | \  > \ \eps \, \right )   \ = \
 0,
  \]
uniformly in $p_r$.    Recalling finally that
\[
\frac{     \mathbf{E}(\sum \rY_{I, r, p_r})}{\mu(r, p_r)} \ = \ \frac{\eee_{p_r,q}(\rB_{r,n})}{\binom{r-n}{p} \cdot \mu(r, p_r)},
\]
we arrive at the required statement.
 \end{proof}
 
 \section{Betti Asymptotics} \label{Betti.Asymptotics.Section}
 This section contains two applications of Stirling's formula. First, we recast the computations of the previous section  to give asymptotic expressions for the entries of a random Betti table. Then we return to the case of curves, and prove Proposition \ref{Betti.Asymptotics.Curves}  from the Introduction. 
 
 In order to get clean statements for the Betti tables it will be helpful to replace convergence in probability with almost everywhere convergence. So we start with some definitions and observations in this direction. As before we fix once and for all an integer $n \ge 2$. Denote by
 \[   \binom{ [\infty]}{n-1} \]
 the set of all $(n-1)$-element subsets of $\ZZ_{> 0}$, so that
 \[   \binom{ [\infty]}{n-1}  \ = \ \bigcup_{r \ge n-1} \, \binom{[r]}{n-1}.\]
 Then put
 \[   \Omega \ = \ [0,1]^{ \binom{ [\infty]}{n-1}}.\]
 This is a countable product of copies of the unit interval, and there are natural projections
 \[  \rho_r : \Omega \ \lra \ \Omega_r. \]
 By a standard procedure (Kolmogorov's Extension Theorem, \cite[Thm. A.3.1]{Durrett}) there is a unique probability measure on $\Omega$ compatible with pull-backs of the standard measures on the $\Omega_r$. Via composition with $\rho_r$, the various functions considered in the previous section -- notably $\rX_I$, $\rY_{I, r,p}$ and  $k_{p,q}(\rB_r)$ -- determine measureable functions on $\Omega$. In other words, all these quantities are random variables on $\Omega$, and the computations of expectations and probablities carried out in  the previous section remain valid in this new setting.  As a matter of notation, for $x \in \Omega$, we write $x_r = \rho_r(x) \in \Omega_r$, and set \[  k_{p,q}(x; r) \ =_{\text{def}} \ k_{p,q}\left(\rB_r\right)\left(x_r \right). \]
 Thus $k_{p,q}(x;r)$ is the indicated entry of the finite  Betti table determined by the Boij-S\"oderberg coefficient vector $x_r \in \Omega_r$. 
 
 The computations of the previous section then lead to the following:
 \begin{theorem} \label{kpq.Thm.2.1} Fix a weight $q \in [1,n]$, and a sequence $\{ p_r \}$ with
 \[  c\   \le \ p_r \  \le \ 1-c \]
 for some small $c > 0$. Then for almost all $x \in \Omega:$ 
 \[
 \lim_{r \to \infty} \ \frac{ k_{p_r,q}(x; r)}{\binom{r-n}{p_r} \cdot \mu(r, p_r) } \ = \ 1. 
 \] 
 \end{theorem}
 \begin{proof}[Sketch of Proof]
 Assuming first that $n \ge 3$, one can simply modify slightly the first part of the proof of \cite[Theorem 5.1.2]{Chung}. In fact, keeping the notation of the previous section, put
 \[
 \rZ_r \ =_{\text{def}} \frac{ \sum_{I \in \binom{[r]}{n-1}} \big( \rY_{I, r, p_r} - \mathbf{E}(\rY_{I, r, p_r}) \big)}{\mu(r, p_r)}.
 \]  
 In view of equations \eqref{kpq.from.Y} and \eqref{Exp.Value.Sum.Y.Eqn}, it suffices to show that $\rZ_r \to 0$ almost everywhere on $\Omega$.   But when $n \ge 3$, it follows directly from \eqref{Prob.Bound} that given  $\delta > 0$:
 \[ \sum_r \, \PP \big( \left | \rZ_r \right |  >  \delta \big) 
\ < \ \infty,
 \]
 and as in \cite{Chung} this implies the desired convergence.  When $n = 2$ this argument does not work because  the right-hand side of \eqref{Prob.Bound} then has order $r^{-1}$. However in this case one can for example adapt in a similar fashion the proof (cf. \cite[Theorem 2.3.5]{Durrett}) of the strong law of large numbers for   independent random variables with finite second and fourth moments. We leave details to the reader. \end{proof}
 
 Stirling's formula now implies the following, which in particular proves Corollary \ref{Stirling.Cor.Intro} from the Introduction:
 \begin{corollary} \label{Stirling.Corollary}
 Fix $n$ and $q$, and let $\{ p_r \}$ be a sequence such that
 \[  p_r \ \to \ \frac{r}{2} \, + \, a \cdot \frac{\sqrt{r}}{2} \]
 for some $a \in \RR$. 
 Then for almost all $x \in \Omega : $
 \[
\left(
\frac{  (q-1)! \cdot (n-q)! \cdot \sqrt{2 \pi r}}{2^{r+2 -3n} \cdot  r^{n-1}}  
\right) \cdot k_{p_r, q}(x; r) \ \to \ e^{-a^2/2}.
 \]
 \end{corollary}

\begin{remark} To simplify the formulas a bit,  we have chosen to break the symmetry between $(p,q)$ and $(r-n-p, n+1-q)$. One could obtain a similar formula for \[p_r \lra \frac{r-n}{2} + b \cdot \frac{ \sqrt{r-n}}{2}. \]\end{remark}
 
  \begin{proof}[Proof of Corollary \ref{Stirling.Corollary}]  First, note that
 \[  \binom{r-n}{p} \ = \ \binom{r}{p} \cdot \frac{ (r-p-n+1) \cdot \ldots \cdot (r-p-1)\cdot (r-p)}{(r-n+1) \cdot \ldots \cdot (r-1) \cdot r},\]
 and hence
 \[
 \lim_{r \to \infty} \frac{  \binom{r-n}{p_r}}{\binom{r}{p_r}} \ = \ \frac{1}{2^n}.
 \]
 Next, observe that
 \[ 
 \lim_{r \to \infty} \frac{ \mu(r, p_r)}{r^{n-1}} \ = \ \frac{1}{2^{2n-1} \cdot (q-1)! \cdot (n-q)!}. \tag{**}\]
 Finally, recall (cf. \cite[Theorem 3.1.2]{Durrett}) that 
 \[ \lim_{r \to \infty} \, \frac{ \sqrt{2 \pi r}}{2^{r+1}} \cdot \binom {r}{p_r} \ = \ e^{-a^2/2}. \tag{***} \]
 The Corollary follows upon putting these together with Theorem \ref{kpq.Thm.2.1}.
 \end{proof}
   
 Changing gears slightly, we conclude this section by working out the asymptotics for the Betti numbers of large degree embeddings of curves. In particular, we prove Proposition \ref{Betti.Asymptotics.Curves} from the Introduction. 
 
 Let $X$ be a smooth projective curve of genus $g$, and let $L_d$ be a divisor of degree $d \gg 0$ on $X$, so that 
 \[   r_d \ =_{\text{def}} \ \hh{0}{X}{L_d} - 1 \ = \ d - g. \] The first point is to compute $k_{p,1}(X; L_d)$. This appears   in \cite[Prop. 3.2]{Teixiodor}, but for the benefit of the reader we recall the derivation. 
In fact,  let $M_d$ be the rank $r_d$ vector bundle on $X$ defined as the kernel of the evaluation mapping
 \[    \text{ev} \, : \, \HH{0}{X}{L_d} \otimes \OO_X \lra L_d. \]
 Then, a well-known argument with Koszul cohomology (cf \cite[\S 3]{ASAV}) shows that
 \begin{align*}
 k_{p,1}(X; L_d) \ &= \ h^0(\Lambda^p M_d \otimes L_d) \, - \, \dim \Lambda^{p+1} H^0(L_d) \\
 k_{p-1,2}(X; L_d) \ &= \ h^1(\Lambda^p M_L \otimes L_d).\end{align*}
 Furthermore, 
 \[
 k_{p-1,2}(X, L_d) \ = \ 0 \ \ \text{for} \ p \le \ r_d - g
 \]
 thanks to a theorem of Green \cite[Theorem 4.a.1]{Kosz1}.
Thus for $p \le r_d - g$:
 \[
  k_{p,1}(X; L_d)  \ = \ \chi\big(X, \Lambda^p M_d \otimes L_d\big)  \, - \, \binom{r_d + 1}{p+1}. \tag{*}\]
 On the other hand, the Euler characteristic in (*) can be computed by Riemann-Roch. In fact, $\mu(M_d) = \frac{-d}{d-g}$, and hence
 \begin{align*}  
  \chi\left(\Lambda^p M_d \otimes L_d\right) \ &= \ \rk(\Lambda^p M_d) \cdot \Big( p \cdot \mu(M_d) + \mu(L_d) + 1-g \Big)\\
 &=  \binom{d-g}{p} \left( \frac{-pd}{d-g} + (d+1 - g) \right).
   \end{align*}
Writing
 \[
 \binom{r_d + 1}{p+1} \ = \ \binom{d-g}{p} \cdot \frac{d + 1 - g}{p+1},
 \]
 we find finally that
 \[
k_{p,1}(X; L_d)  \ = \ \binom{d-g}{p} \left( \frac{-pd}{d-g} + (d+1 - g) - \frac{d+1-g}{p+1} \right) 
 \]
provided that $p \le r_d - g$.  Proposition \ref{Betti.Asymptotics.Curves} from the Introduction then follows as in the proof of Corollary \ref{Stirling.Corollary}
 by an application of Stirling's formula.  
 
 \begin{remark} \label{Curve.Variability.Remark}
 When $p > r_d - g = d -2g$, the Betti number $k_{p,1}(X; L_d)$ depends on $X$ and $L_d$ in a manner that is not yet completely understood. On the other hand, the computation of syzygies as Koszul cohomology groups shows that $K_{p,1}(X; L_d)$ is a sub-quotient of $\Lambda^{p} H^0(L_d) \otimes H^0(L_d)$, and therefore
 \[
 k_{p,1}(X; L_d) \ \le \ \binom{d + 1 - g}{p} \cdot (d + 1- g).
 \]
 For $p > d - 2g$ this becomes a polynomial upper bound that is much smaller than the value of $k_{p,1}$ for $p \approx \frac{d-g}{2}$. Thus as far as asymptotics are concerned, the $K_{p,1}$ with $p > d - 2g$ are essentially negligible.    Of course in the statement of Proposition  \ref{Betti.Asymptotics.Curves} these ambiguities don't even enter the picture.  \end{remark}
 
   \section{Complements and Open Questions} \label{Complements.Section}
    
   In this section we establish some related results, and propose some open problems for further research.

   \subsection{Weight Functions} \label{Weight.Function.Subsection}
   Here we discuss the possibility of generalizing our results by taking non-uniform probability measures on $\Omega_r$. We will assume for simplicity that $n = 2$: in this case pure diagrams are labeled by a single integer $i \in [1, r] = \binom{[r]}{1}$. 
   
Fix  a function 
   \[  h : [0,1] \lra [0,1]. \]
   Then we can use $h$ to define the \textit{weighted} random Betti table
   \[  \rB_r^h \ =_{\text{def}} \sum_{i=1}^r  h(\tfrac{i}{r})\cdot \rX_i \cdot \pi(r, i), \]
   where $\rX_i$ $(1 \le i \le r)$ as above are independent random variables uniformly distributed on $[0,1]$. Equivalently, we are weighting the pure diagram $\pi(r,i)$ with a random variable uniformly distributed on the interval $[0, h(\tfrac{i}{r})]$.        This procedure is illustrated in Figure 4, which  shows a random Betti table with $r = 500$ weighted according to the somewhat whimsical choice 
   \[ h(t) \ =\ \sin^2(2 \pi( t-.35)).\] The plot on the left displays  random coefficients $x_i$  ($1 \le i \le 500$) chosen uniformly in the interval $[0, h(i/500)]$. The resulting Betti numbers $k_{p,1}$ are plotted on the right.
   
\begin{figure} \label{Random.Diagram}
 \hskip -10pt     \includegraphics[scale = .85]{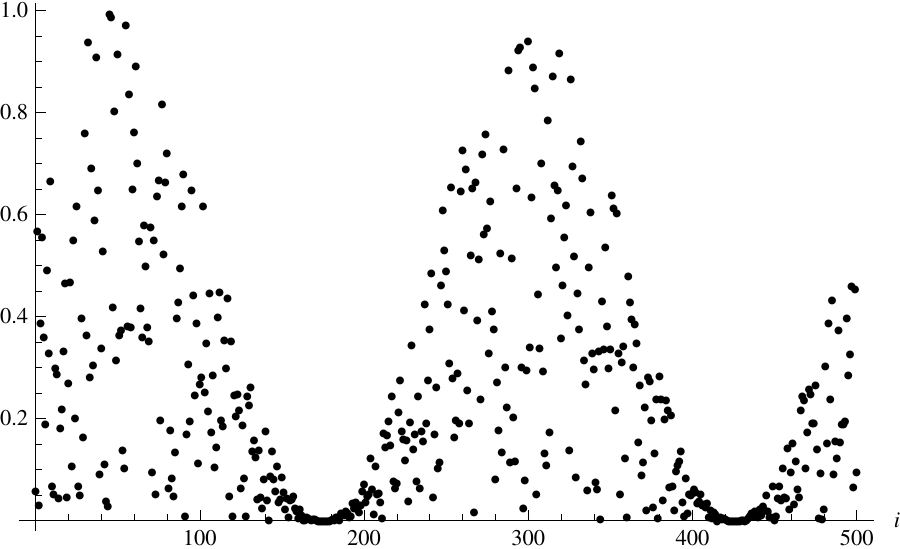} \qquad  \includegraphics[scale = .85]{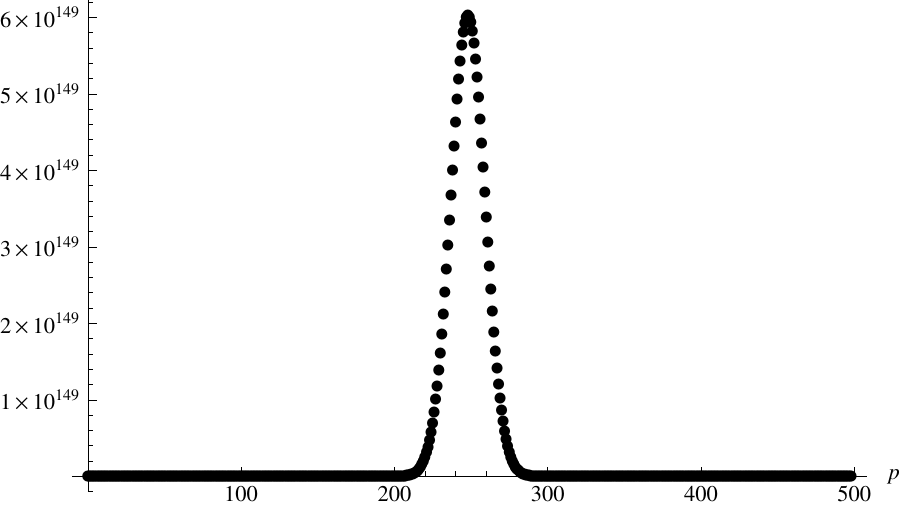}
   \vskip -10pt
   \caption{Weighted random Betti table with $h(t) = \sin^2(2 \pi (t-.35))$ and $r = 500$}
   \vskip 10pt
  \end{figure}
   
   Now for  arbitrary $h$ one can't expect the rows of $\rB_r^h(x)$ to be normally distributed for most $x$. For example  if $\text{supp}(h) \subset [0, \frac{1}{2}]$, then
   \[  k_{p,1}\big(\rB_r^h(x)\big) \ = \ 0 \]
   for every $x$ whenever $p > \frac{r}{2}$.     However if $h$ is smooth, and if one rules out the sort of problem just illustrated, then the qualitative results established in the previous section  -- which correspond to the case $h(t) \equiv 1$ -- do remain valid.
    \begin{theorem} \label{Weighted.Table.Thm}
    Assume that $h : [0,1] \lra [0,1]$ is smooth, and that it is not identically zero on $[0, \tfrac{1}{2}]$ or on $[\tfrac{1}{2}, 1]$. There exist functions $F_1(r), F_2(r)$ with the property that if  $p_r$ is a  sequence with 
    \[  p_r \to \frac{r}{2} + a \cdot \frac{\sqrt{r}}{2} \]
    for some $a \in \RR$, then 
        \[  F_1(r) \cdot k_{p_r,1}\big(\rB^h_r(x)\big) \to e^{-a^2/2}\ \ ,  \ \ F_2(r) \cdot k_{p_r,2}\big(\rB^h_r(x)
        \big) \to e^{-a^2/2} \]
    for almost all $x \in \Omega$.
    \end{theorem}
    \noi This suggests that the picture established in  the previous section is actually quite robust. 
    
    \begin{remark}
    Alternatively, one could consider the ``deterministic" Betti diagram
  \[    \sum_{i =1}^r \, h(\tfrac{i}{r}) \cdot \pi(r, i). \]
An evident analogue of the theorem remains true in this setting. See Problem \ref{NIce.BS.Function.Question} for  the potential interest in statements of this sort.
    \end{remark}
    
    \begin{remark} 
  When $n \ge 3$ one expects that a similar result holds  if one uses a smooth function \[ h: [0,1]^{n-1} \lra [0,1]\] to weight $\rX_I$ ($I \in \binom{[r]}{n-1}$) by $h(\tfrac{i_1}{r}, \ldots, \tfrac{i_{n-1}}{r})$. 
    \end{remark}
    
    \begin{proof}[Sketch of Proof of Theorem \ref{Weighted.Table.Thm}]     We will assume $q = 1$, the case $q = 2$ being similar.
   Set
   \begin{equation}\label{eqn:H integral}
H(r,p)\ = \ \frac{ \int_{\frac{p+1}{r}}^{1} \, h(t) \cdot (t- \frac{p+1}{r}) \ d{t}}{2(1-\frac{p+1}{r})}.
\end{equation}
Observe that $H(r, \tfrac{r}{2}) \ne 0$  for $r \gg 0$ thanks to the fact that $h$ is smooth and does not vanish identically on $[\tfrac{1}{2},1]$. Furthermore,  if $p_r \to \frac{r}{2} + a \cdot \frac{\sqrt{r}}{2} $ then $H(r, p_r) \lra H(r, \tfrac{r}{2})$. Now, as in the proof of Theorem \ref{Expected.Value.Thm}: 
\begin{align*}
\mathbf{e}_{p,1}(\rB^h_{r})\ &=\ \binom{r-2}{p}\cdot \frac{\sum_{i=1}^{r} h(\tfrac{i}{r})(i-p-1)_+}{2(r-p-1)}\\
&=\ \binom{r-2}{p}\cdot \frac{\sum_{i=p+2}^{r} h(\tfrac{i}{r})(i-p-1)}{2(r-p-1)}\\
&=\ \binom{r-2}{p}\cdot \left(\frac{1}{2\left(1-\tfrac{p+1}{r}\right)}\right) \cdot \left(\sum_{i=p+2}^{r} h(\tfrac{i}{r})(\tfrac{i}{r}-\tfrac{p+1}{r})\right) .
\end{align*}
But the product of the two righthand terms is a Riemann sum for $r \cdot H(r,p)$, and  one checks  that with $p_r$ as above one in fact has
\[
\frac{\mathbf{e}_{p_r,1}(\rB^h_{r})}{\binom{r-2}{p_r} \cdot r \cdot H(r,\tfrac{r}{2})} \, \to \, 1 \quad \text{ as } \quad r\to \infty.
\]
  On the other hand, as in Theorem \ref{kpq.Thm.2.1}, 
\[ \frac{k_{p_r,1}\big( \rB^h_r(x)\big)}{\mathbf{e}_{p_r,1}(\rB^h_{r})}\,  \to\, 1\]
for almost all $x \in \Omega$. 
One then concludes with an application of Stirling's formula as in the proof of Corollary \ref{Stirling.Corollary}. 
   \end{proof}

   \subsection{Some generic $2$-regular modules} 
   \label{Generic.Mxs.Subsection}
   As noted in the Introduction, it is interesting to compare the ``random" Betti tables studied here with those  arising from modules that are generic in a more traditional sense.  Once again we will assume that $n = 2$, where the algebraic situation  is particularly clear.
   
   Consider then a finite-length graded module $M$  over the polynomial ring $R=k[x_1, \ldots, x_{r-2}]$,   with
  \[    k_{p,q}(M) \ = \ 0 \ \ \text{for } \, q \ne 1\, , \, 2. \]
 Such a module is $2$-regular, and  so is given by two vector spaces $M_1$, $M_2$, say of dimensions $m_1$, $m_2$,  together with a mapping:
 \[   M_1 \otimes R_1 \lra M_2 \tag{*}  \]
 determining the $R$-module structure. After choosing bases, (*)  is in turn equivalent to specifying an $m_2 \times m_1$ matrix $\phi$ of linear forms. We will write $M_\phi$ for the module corresponding to a matrix $\phi$.    Note that the possible choices of $\phi$ are parametrized by an irreducible variety.

Returning for a moment to the setting of \S \ref{Random.Diagrams.Section}, fix now a Boij-S\"oderberg coefficient vector  
 \[  x = (x_1, \ldots, x_r) \ \in \ \Omega_r \, = \, [0,1]^r, \]
 and consider the corresponding Betti table. This is a sum of $r$ pure tables, and the expected value of each $x_i$ is $\tfrac{1}{2}$. Thus the expected formal multiplicity of the table in question is $\tfrac{r}{2}$, and the expected formal Hilbert function has values $\tfrac{r}{4}$ in degrees $1$ and $2$, and  $0$ for all other degrees.  This leads us to consider modules $M$   with
 \[   \dim M_1 \, = \, \dim M_2 \, = \, s  \]
for some integer $s$, which  are described by an $s \times s$ matrix $\phi$ of linear forms.  
   
 We prove:
  \begin{proposition} \label{Table.Generic.Mx.Prop}
  There are arbitrarily large integers $s$ with the property that  if 
  $\phi$ is a general $s\times s$ matrix of linear forms, then the Betti table of $M_{\phi}$ is a sum of pure tables of type
\[  \pi\left( r,  \lfloor \tfrac{r+1}{2}\rfloor\right) \ \ \text{and} \ \ \pi\left( r,  \lceil \tfrac{r+1}{2}\rceil\right).  \]
In particular, if $r$ is odd then $M_\phi$ has a pure resolution.
  \end{proposition} 
\noi We note that a conjecture of Eisenbud--Fl\o ystad--Weyman \cite[Conjecture 6.1]{EFG} implies that the statement should hold for all sufficiently large $s$. In any event,   the Proposition shows that genericity in the module-theoretic sense can lead to completely different behavior than that which occurs for the random tables considered above.  Observe that this does not contradict Conjecture \ref{Betti.Asymptotics.Conj}: in fact,   the results of \cite{ASAV} imply  that the resolutions arising in the geometric setting are very far from pure.
   
   \begin{proof}[Proof of Proposition \ref{Table.Generic.Mx.Prop}]
   For any given $s$, it suffices (by the semicontinuity of Betti
numbers) to produce one example where the statement holds.  We will
prove the theorem when $s=r-1$.  Then by taking direct sums, this will
imply the statement for any $s$ that is a multiple of $r-1$.  We
henceforth assume that $s=r-1$.

Using the notation of \cite[A2.6]{Eisenbud}, we consider the direct sum of
free resolutions 
\[
\left(\mathcal C^{\lfloor \frac{r-3}{2}
\rfloor}\oplus \mathcal C^{\lceil \frac{r-3}{2} \rceil}\right)
\otimes_R R(-1)\] derived from a general map $R^{r-1}(-1)\to R^2$.  The
module $M$ resolved by this free complex has regularity $2$ and
satisfies $\dim M_1=\dim M_2=r-1$.  Since the Betti table of the
complex $\mathcal C^{i}\otimes_R R(-1)$ equals $(r-1)\cdot
\pi(r,{i+2})$, the statement follows immediately.
   \end{proof}

   \subsection{Open Questions}
   We conclude by proposing a few problems.
   
   First, the reader will note that all our arguments are purely numerical in nature -- they don't give an a priori  sense why one would expect   to see normal distibution of Betti numbers.
   \begin{problem}
  Find a probabilistic (or other) model that  explains the behavior that we have established for random Betti tables.
   \end{problem}

Returning to the geometric questions motivating the present work,  consider as in the Introduction a smooth projective variety 
\[
X \ \subseteq  \  \PP^{r_d}
\]  
of dimension $n$ embedded by $L_d = dA + P$.    Assuming for instance that $\HH{i}{X}{\OO_X} = 0$ for $0 < i < n$, so that $X$ is arithmetically Cohen-Macaulay when $d \gg 0$, one can consider the Boij-S\"oderberg decomposition of the resolution of $X$.    We pose the somewhat vague
\begin{problem} \label{NIce.BS.Function.Question}
Can one find a ``nice" function $h$    that governs this decomposition as in \S\ref{Weight.Function.Subsection}? If so, what are its properties?
\end{problem}
\noi  The problem is most immediately meaningful in the case $n =2$, in which case  Boij-S\"oderberg coefficients $x_i$ are well-defined for each $i \in [1, r_d -2]$. The question then becomes whether they interpolate a fixed smooth or continuous function $h$.\footnote{It is   possible that the scaling  chosen   above is not the appropriate one to use.} At the moment an affirmative solution seems out of reach, since a good answer to the Problem would presumably imply Conjecture \ref{Betti.Asymptotics.Conj}. On the other hand, the question is philosophically   in keeping with recent work on linear series (eg \cite[Chapter 2.2.C]{PAG}, \cite{ELMNP}, \cite{CBALS}) and Hilbert series (eg \cite{BeckStap}, \cite{McCabeSmith}), where it has become apparent that one can often define asymptotic invariants that behave surprisingly well. The asymptotic Boij-S\"oderberg coefficients of large degree embeddings of curves have been analyzed by the second author \cite{Erman}.
   
   Finally, we expect that Conjecture \ref{Betti.Asymptotics.Conj}, if true, is esentially an algebraic fact.
   \begin{problem}
Find  a purely algebraic statement that implies, or runs parallel to, Conjecture \ref{Betti.Asymptotics.Conj}.
   \end{problem}

 %
 %
 %
 %

 \end{document}